\documentclass[10pt]{article}
\usepackage{amsmath,amsthm,amssymb,graphicx,color}
\usepackage{geometry}
\usepackage[dvipsnames]{xcolor}

\usepackage{amsmath,amsthm,amssymb,tikz,epsfig,amsfonts,latexsym,  bbm, mathrsfs,natbib,hyperref}
\usepackage{epsfig,amsfonts,latexsym, MnSymbol}

\usepackage{authblk}

\newcommand{\inv}{{-1}}

\newcommand{\mbbz}{\mathbb{Z}}
\newcommand{\mbbn}{\mathbb{N}}
\newcommand{\mcala}{\mathcal{A}}

\newcommand{\eqd}{\stackrel{d}{=}}

\DeclareMathOperator{\prob}{\mathbb{P}}

\DeclareMathOperator{\exptn}{\mathbb{E}}

\DeclareMathOperator{\card}{Card}

\newcommand{\beq}{\begin{equation}}
\newcommand{\eeq}{\end{equation}}
\newcommand{\alns}[1]{\begin{align*}#1\end{align*}}
\newcommand{\aln}[1]{\begin{align} #1 \end{align}}
\newcommand{\been}{\begin{enumerate}}
\newcommand{\een}{\end{enumerate}}

\newcommand{\textred}[1]{\textcolor{black}{#1}}
\newcommand{\bbz}{\mathbb{Z}}
\newcommand{\cala}{\mathcal{A}}
\newcommand{\bbn}{\mathbb{N}}

\usepackage{cleveref}
\newtheorem{thm}{Theorem}[section]
\newtheorem{propn}[thm]{Proposition}

\theoremstyle{remark}
\newtheorem{remark}[thm]{Remark}
\newtheorem{cor}[thm]{Corollary}

\theoremstyle{definition}

\hypersetup{
    colorlinks=true,
    linkcolor=ForestGreen,
    citecolor=blue,
    urlcolor=blue,
    pdfborder={0 0 0}
}

\numberwithin{equation}{section}

\title{Slower variation of the generation sizes induced by heavy-tailed environment for geometric branching}
\date{\today}

\author[1]{Ayan Bhattacharya}
\author[2]{Zbigniew Palmowski}
\affil[1]{Centrum Wiskunde \& Informatica, Amsterdam, The Netherlands}
\affil[2]{Faculty of Pure and Applied Mathematics, Wroc\l aw University of Science and Technology, Wroc\l aw, Poland}

\allowdisplaybreaks

\begin{document}

\maketitle

\begin{abstract}
Motivated by seminal paper of Kozlov et al. \cite{kesten:kozlov:spitzer:1975} we consider in this paper a
branching process with a geometric offspring distribution parametrized by random success probability $A$ and immigration equals $1$ in each
generation. In contrast to above mentioned article, we assume that environment is heavy-tailed, that is $\log A^\inv (1-A)$ is regularly varying with a
parameter $\alpha>1$, that is that $\prob \Big( \log A^\inv (1-A) > x \Big) = x^{-\alpha} L(x)$ for a slowly varying function $L$.
We will prove that although the offspring distribution is light-tailed, the environment itself can produce
extremely heavy tails of distribution of the population at $n$-th generation which gets even heavier with $n$ increasing.
Precisely, in this work, we prove that asymptotic tail $\prob(Z_l \ge m)$ of $l$-th population $Z_l$ is of order $
\Big(\log^{(l)} m \Big)^{-\alpha} L \Big(\log^{(l)} m \Big)$
for large $m$, where $\log^{(l)} m = \log \ldots \log m$.
The proof is mainly based on Tauberian theorem.
Using this result we also analyze the asymptotic behaviour of the first passage time $T_n$ of the state $n \in \mbbz$ by the walker
in a neighborhood random walk in random environment created by  independent copies $(A_i : i \in \bbz)$ of  $(0,1)$-valued random variable  $A$.
\end{abstract}



\section{Introduction and main results}


We consider branching process appeared in \cite{kesten:kozlov:spitzer:1975} to study limit theorems for hitting times associated to the Random Walk in Random Environment (RWRE). We describe briefly the model RWRE and the associated geometric Branching Process in Random Environment (BPRE). Consider a collection $(A_i : i \in \bbz)$ of i.i.d. (independently and identically distributed) $(0,1)$-valued random variables. Let $\cala$ be the natural $\sigma$-field associated to the collection $(A_i : i \in \bbz)$. Let $(X_k : k \in \bbn)$ be  a collection of $\bbz$-valued random variables such that, $X_0 =0$
\alns{
\prob \big( X_{k+1} = X_k + 1 \big| \cala, X_0 = i_0, \ldots, X_k = i_k \big) = A_{i_k} = 1- \prob \big( X_{k+1} = X_k -1 \big| \cala, X_0 = i_0, \ldots, X_k = i_k  \big)
}
for all $i_j \in \bbz$, $1 \le j \le k$ and  $k  \ge 1$. The collection $(A_i : i \in \bbz)$ is called the random environment.
For this random walk \cite{kesten:kozlov:spitzer:1975}
studied asymptotic distribution (after appropriate normalization) of a sequence of  hitting times
$T_n= \inf\{k > 0: X_k =n\}$ of the state $n \in \mbbz$ by the walker in the random environment.  Following the arguments \textred{(see after Remark~3 in page~148 of \cite{kesten:kozlov:spitzer:1975})} given in the aforementioned work, we have
\begin{equation}
T_n = n + 2 \sum_{i = -\infty}^{\infty} U_i^{(n)}\label{mainrepr}
\end{equation}
where $U_i^{(n)} := \card \{ k < T_n : X_k = i, X_{k+1}= i-1\}$ denotes the number of times moved left being at state $\{i \}$ with $\card(K)$ is the cardinality of the set $K$. Under the following assumptions (see assumption (1.2) in \cite{kesten:kozlov:spitzer:1975}) on the environment
\aln{
\exptn \bigg( \log \frac{1-A}{A} \bigg) < 0 ~~~~~~ \mbox{but} ~~~~~~ \exptn \Big( \frac{1-A}{A} \Big) \ge 1,  \label{eq:assumption:solomon}
}
it follows that $X_k \to \infty$ almost surely as $k \to \infty$. So $\sum_{i= - \infty}^{0} U_i^{(n)}$ is finite almost surely and can be ignored in asymptotic analysis of $T_n$. It is also easy to observe that $\sum_{i = n+1}^\infty U_i^{(n)} = 0$ almost surely as \textred{the walker can not reach to $i$ before hitting $n$ for all $i \ge n+1$.} Thus the asymptotic behavior of $T_n$ is solely determined by the asymptotic behavior of $\sum_{i = 1}^n U_i^{(n)}$. The following observation
\aln{
\sum_{i=1}^n U_i^{(n)} \eqd \sum_{l=0}^{n-1} Z_l.\label{disteq}
}
has been used in \cite{kesten:kozlov:spitzer:1975} to derive the asymptotics of $\sum_{i=1}^{n-1} U_i^{(n)}$, where $Z_n$ denotes the size of the $n$-th generation of a BPRE with one immigrant in each generation. The BPRE is constructed in such a way that
\aln{
Z_n = \sum_{i=1}^{Z_{n-1} + 1} B_{n,i},\label{Zn}
}
where  $(B_{n,i} :  i \ge 1)$ are independent copies of the geometric
random variable $B_n$ such that
\aln{
\prob (B_n = k) = A_{n-1} \Big( 1- A_{n-1} \Big)^k \mbox{ for all } k \ge 0,~ n\ge 1 \label{eq:prob:mass:func:Bn}
}
conditioned on $\mcala$. \cite{kesten:kozlov:spitzer:1975} derived central limit theorem for $n^{-1/\kappa}T_n$ \textred{if there exists a  $\kappa >0$} such that 
\aln{
\exptn \bigg( \exp \bigg\{ \kappa \log \frac{1-A}{A} \bigg\} \bigg) =1. \label{eq:assumption:kks}
}
Note that the assumption in \eqref{eq:assumption:kks} implies that the random variable $\log A^\inv (1-A)$ has an exponentially decaying right tail. Under above assumptions, after appropriate scaling, $T_n$ has the same asymptotical tail like scaled \textred{$\sum_{l=0}^{n-1} Z_l$}
and  converges to a $\kappa$-stable random variable if $\kappa \in (0,2)$ and Gaussian random variable if $\kappa \ge 2$ \textred{(see main result in \cite{kesten:kozlov:spitzer:1975})}.

%

The aim of this article is to study the asymptotic behavior of branching process $Z_n$ under the assumption that $\log A^\inv (1-A)$ has a regularly varying (instead of exponentially decaying) tail. We assume then that
\aln{
\prob \Big( \log A^\inv (1-A) < x \Big) =  \begin{cases} 1- x^{-\alpha}  L(x) & \mbox{ if } x >\eta \\ G(x) & \mbox{ if } x < \eta \end{cases} \label{assregvar}
}
 for some $\eta > 0$ and $\alpha >1$  where $L(\cdot)$ is a slowly varying function i.e. $\lim_{x \to \infty} L(tx)/ L(x) = 1$. We assume that $G$ is chosen in such a way that \eqref{eq:assumption:solomon} holds. \textred{Note that $\log A^\inv (1-A)$ is a real-valued random variable. In \eqref{assregvar}, we only put restrictions on the right-tail of the distribution of $\log A^\inv (1-A)$ but did not assume anything about the left-tail. It is clear that the probability of the walker moving to right is small if the value of $A$ is close to $0$ which causes large values of $T_n$. It is clear that $\log(a^{-1}(1-a))$ is a decreasing function of $a \in (0,1)$. So the tail behavior of $A$ near $0$ is same as  the tail behavior of $\log A^{-1}(1- A)$ near infinity. So large values of $\log A^\inv (1 - A)$ causes large values of $T_n$. As we are interested in the probability of large values of $T_n$, right-tail of the random variable $\log A^\inv (1-A)$ only matters.  Note that \eqref{eq:assumption:solomon} implies that the BPRE under consideration is subcritical without immigrant and hence becomes extinct eventually for almost all environments. The formula  for $T_n$ involves the first $n$ generations of the subcritical BPRE. So there is a positive probability that the extinction of BPRE may happen before generation $n$. The immigration is important for survival of the tree  till generation $n$. But it does not contribute too much the large values of $T_n$ as it is constant through out all the generations. There is another interpretation of the immigrant. Note that we are considering here nearest-neighbour random walk on ${\mathbb Z}$ and so the walker has to spend atleast one unit of time at each state $i$ before hitting $n$ for all $i = 0,1, 2, \ldots, n-1$. Hence one immigrant in each generation appears in the description (see \eqref{Zn}) of the BPRE.}



 Following \cite{vatutin:dyakonova:sagitov:2013}, if there exists $\beta >0$, \textred{given by} the following equation
\begin{align*}
\exptn \Big[ \exp \Big\{ \beta \log \exptn(Z_1| \mathcal{A}) \Big\} \log \exptn(Z_1| \mathcal{A}) \Big] = 0
\end{align*}
which becomes
$$ \exptn \Big[ \exp \Big\{ \beta \log \frac{1-A}{A} \Big\} \log \frac{1-A}{A} \Big] = 0 $$
in our case; then the asymptotic behavior of BPRE crucially depends on the parameter $\beta$. \textred{ The parameter $\beta$  may not exist always. We would like to underline the fact that in our case, the right-tail of $\log A^{-1}(1 - A)$ is regularly varying stated in \eqref{assregvar} and so there does not exist any $\beta>0$ such that 
$$\exptn \Big[ \Big|\exp \Big\{ \beta \log \frac{1-A}{A} \Big\} \log \frac{1-A}{A} \Big| \Big] < \infty.$$ }
 We are interested in the annealed behavior of the generation sizes $(Z_n : n \ge 1)$ of the BPRE in this paper. Our first result Theorem~\ref{thm:tail:first:generation:size} shows that $\prob(Z_1 \ge m) \sim (\log m)^{-\alpha} L(\log m)$ and hence have slowly varying tail. It is clear that $\exptn(Z_1) = \infty$. We would also like to stress the fact that this behavior is not totally unexpected. Note that the tail behavior of $n^{- 1/ \kappa} T_n$ is regularly varying if $\kappa \in (0,2)$ as the limit is  stable random variable under the assumption stated in \eqref{eq:assumption:kks}. So it is natural to guess that $Z_1$ have slowly varying tail under the assumption \eqref{assregvar} though the form of the slowly varying function is far from being obvious. We derive exact form of the slowly varying function in Theorem~\ref{thm:tail:first:generation:size} for $Z_1$ and Theorem~\ref{thm:tail:general:generation:size} for $Z_l$ with $l \ge 2$.  These results are used finally to derive the asymptotics for $T_n$ in Theorem~\ref{thm:tail:Tn}. To the best of our knowledge, this kind of example in BPRE is missing in the literature where generation sizes have exponentially decaying tail given the environment but have slowly varying tail after averaging out the effect of random environment.
As a consequence of slowly varying tail of $Z_1$, it is easy to guess that the annealed behavior of generation sizes is very similar to a GW tree with infinite mean. Branching process with infinite mean is well-studied in literature and a brief review indicating contribution of this article in that literature is given after stating main results of this paper.

\begin{thm} \label{thm:tail:first:generation:size}
Under the assumptions \eqref{eq:prob:mass:func:Bn}, \eqref{eq:assumption:solomon} and \eqref{assregvar},
\alns{
\lim_{m \to \infty} \frac{\prob(Z_1 > m)}{(\log m)^{-\alpha} L(\log m)} =1.
}
\end{thm}

\begin{thm} \label{thm:tail:general:generation:size}
Under the assumptions \eqref{eq:prob:mass:func:Bn}, \eqref{eq:assumption:solomon} and \eqref{assregvar},
\aln{
\lim_{m \to \infty} \frac{\prob (Z_l > m)}{(\log^{(l)} m)^{-\alpha} L(\log^{(l)}m)} = \alpha^{-\alpha} \label{eq:tail:asymp:generation:size}
}
 for $l \ge 2$ where $\log^{(l)} m = \underbrace{\log \ldots \log}_{l~~many} m$.
\end{thm}

\begin{cor}Under the assumptions in Theorem~\ref{thm:tail:general:generation:size}, we have 
\begin{align*}
\lim_{l \to \infty} \lim_{m \to \infty} \frac{\prob \big( \log^{(l)} Z_l > m \big)}{ m^{-\alpha} L(m)} = \alpha^{-\alpha}.
\end{align*}
\end{cor}

Theorem \ref{thm:tail:first:generation:size} shows that the tail of $Z_1$ is surprisingly heavy and it is slowly varying. 
What is more surprising, with each new generation is getting even more heavy and the tail is slowly varying.
What should be underlined, this type of behaviour is a consequence of an environment only, not
branching mechanism which is of geometric type. 
In our opinion it is first time that such unusual behaviour has been observed in the context of branching processes.  As a consequence of slowly varying tail of $Z_1$, annealed behavior of the considered branching process seems to be similar to the branching processes with infinite mean (see \cite{seneta:1973}, \cite{hudson:seneta:1977}, \cite{davies:1978}, \cite{grey:1977}, \cite{cohn:1977}, \cite{schuh:barbour:1977} for example). The asymptotic study in this paper is different as we are studying the asymptotics by looking at the tail behavior of the generation sizes rather than their probability generating functions. 

As a corollary we can get another very important result
concerning the first passage time $T_n$ of the state $n \in \mbbz$ by the walker
in a \textred{nearest neighbour random walk} in random environment created by  i.i.d. $(0,1)$-valued random variables $(A_i : i \in \bbz)$
with generic $A$.


\begin{thm} \label{thm:tail:Tn}
Under the assumptions \eqref{eq:assumption:solomon}, \eqref{eq:prob:mass:func:Bn} and \eqref{assregvar},
\aln{
\lim_{m \to \infty} \frac{\prob \bigg( \log^{(n-1)} \Big[ 2^{-1}\big( T_n - n - 2 \sum_{i \le 0} U_i^{(n)} \big)\big] > m \bigg)}{m^{-\alpha} L(m)} = \alpha^{-\alpha}
} for all $n \ge 2$. 
\end{thm}

\begin{remark}
In Theorem \ref{thm:tail:Tn}
we also identify the asymptotic distribution of
the first passage time $T_n$ of the state $n \in \mbbz$ by the walker
in a neighborhood random walk in random environment.
The counterpart of so-called 'scaling' in the central limit theorem
takes the surprising form of taking $n$-times logarithm.
This is also a consequence of heavy-tailed environment.
In this case roughly one needs $\exp_{(n)}$ trials to cross
the barrier created by heavy-tailed environment where $\exp_{(n)}$ is the inverse function of $\log^{(n)}$. Indeed, the
large values of  $\log A^\inv (1-A)$ by \eqref{assregvar}
corresponds to values of $A$ close to $0$. This is related, by single one jump principle,
with the phenomenon that there is a place on a lattice line
that blocks move to the right and hence one has to wait long time to get to the state $n$.
\end{remark}



\begin{proof}[Proof of Theorem~\ref{thm:tail:Tn}]
\textred{It follows from \eqref{mainrepr} and \eqref{disteq}} that it is enough to prove
\begin{align}
 \lim_{m \to \infty} \frac{\prob \Big(  \log^{(n-1)} \big( \sum_{i =1}^{n-1} Z_i \big) > m  \Big)}{\big( m^{-\alpha}  L(m)\big) } = \alpha^{- \alpha}. \label{eq:aim:T_n}
\end{align}
We shall prove it using upper and lower bound of the probability in \eqref{eq:aim:T_n}. Note that
\begin{align}
\prob \big[ \log^{\textred{(n-1)}} \big( \sum_{i =1}^{n-1} Z_i \big) > m \big] \ge \prob \big( \log^{\textred{(n-1)}} Z_{n-1} \ge m \big) \sim \alpha^{-\alpha} m^{-\alpha} L(m)
\end{align}
for all $n \ge 2$ \textred{as $m \to \infty$}  using Theorem~\ref{thm:tail:general:generation:size}. So we are done with the lower bound. We have to prove now the upper bound. We shall first observe that $ \{ \sum_{i=1}^{n-1} Z_i > m\} \subset \cup_{i=1}^{n-1} \{ Z_i > (n-1)^{-1} m \} $. Thus we have
$$\prob \Big( \sum_{l=1}^{n-1} Z_l > m \Big) \le \sum_{l=1}^{n-1} \prob (Z_l \ge (n-1)^{-1} m)$$
for all $n \ge 2$. For large enough $m$, we have
\begin{align}
& \sum_{l=1}^{n-1} \prob \big( Z_l > (n-1)^{-1} m \big) \nonumber \\
&  \textred{\sim} \big[\log m  - \log (n-1)\big)]^{-\alpha} L(\log  m - \log (n-1)) \nonumber \\
& \hspace{2cm}  + \sum_{l=2}^{n-1}  \alpha^{-\alpha} \big[ \log^{(l)} m - \log^{(l)} (n-1) \big]^{-\alpha} L \big[ \log^{(l)} m - \log^{(l)} (n-1) \big] \nonumber \\
& \sim \alpha^{-\alpha} \big( \log^{(n-1)}(m) \big)^{-\alpha}  L(\log^{(n-1)}(m))
\end{align}
for all $n \ge 2$. This implies that
\begin{align*}
\limsup_{m \to \infty} \frac{\prob \Big[ \log^{(n-1)} \Big( \sum_{l=1}^{n-1} Z_l \big) > m \Big]}{m^{-\alpha} L(m)} \le \alpha^{-\alpha}.
\end{align*}
for every $n \ge 2$. Letting $n \to \infty$, we get the second half for \eqref{eq:aim:T_n}.
\end{proof}


\section{Proofs}

\begin{proof}[Proof of Theorem~\ref{thm:tail:first:generation:size}]
Note that $\prob(Z_1 \ge m) = \exptn \Big( (1-A)^m \Big)$. To study the asymptotics of above expectation as $m \to \infty$,
we have to understand the tail behavior of $A$ near $0$. It follows from the assumption \eqref{assregvar} that
\aln{
\prob(A > a) = \begin{cases} G \Big(\log \frac{1-a}{a} \Big) & \mbox{ if } a > \Big( 1 + e^\eta \Big)^\inv \\ 1- R_\alpha \Big( \log \frac{1-a}{a} \Big) & \mbox{ if } 0 < a < (1 + e^\eta)^\inv, \end{cases}
}
where $R_\alpha(a) = a^{-\alpha} L(a)$ for all $a >0$. Hence we obtain the following equation:
\aln{
\prob(Z_1 \ge m) = \int_0^{(1+ e^\eta)^\inv} (1-a)^m d R_\alpha \Big( \log \frac{1- a}{a} \Big) + \int_{(1+e^\eta)^\inv}^\infty (1-a)^m d G \Big( \log \frac{1-a}{a} \Big). \label{eq:tail:zone_disp_1}
}
Using the fact that $(1-a) < e^{\eta}(1+ e^\eta)^\inv$ if $a > (1+e^\eta)^\inv$, we can see that the second integral in \eqref{eq:tail:zone_disp_1} can be bounded by $e^{m \eta} (1+ e^\eta)^{-m}$ which decays exponentially with $m$. It is then enough to consider the first integral. Substituting $ y = \log \Big( \frac{1-a}{a} \Big)$, we obtain the following expression for the first integral in \eqref{eq:tail:zone_disp_1} as
\aln{
\int_{\eta}^\infty \Big( 1- (1+ e^y)^\inv \Big)^{\textred{m}} d R_\alpha(y) = \int_{\eta}^\infty \Big( 1 + e^{-y}\Big)^{-m} d R_\alpha(y).  \label{eq:tail:zone_disp_2}
}
Again substituting $e^u = (1 + e^{-y})$, \eqref{eq:tail:zone_disp_2} can be transformed into
\aln{
\int_0^{\log (1 + e^{-\eta})} e^{-mu} d R_\alpha \Big( - \log(e^u -1) \Big). \label{eq:tail:zone_disp_3}
}
This expression helps to understand the behavior of the integral as $m \to \infty$ since this is the Laplace transform of the measure $R_\alpha( - \log(e^u -1))$ and we can use Tauberian Theorem~1.7.$1'$ in \cite{bingham:goldie:teugels:1987}.
Note that
\aln{
\lim_{u \to 0} \frac{\Big( - \log(e^u -1) \Big)^{-\alpha} L \Big( - \log (e^u -1) \Big)}{ \Big(-\log u \Big)^{-\alpha} L \Big( - \log u \Big) } = 1 \label{eq:tail:zone:Ralpha:zero}
}
and thus
\alns{
\lim_{u \to 0} \frac{R_\alpha \Big( - \log (e^u -1) \Big)}{\big( - \log u \big)^{-\alpha} L(- \log u)}=1.
}
This gives
\aln{
\lim_{m \rightarrow \infty} \frac{ \int_0^{\log (1 + e^{-\eta})} e^{-mu} d R_\alpha \Big( - \log(e^u -1) \Big)}{
\Big( \log m \Big)^\alpha L \Big(\log m \Big)}=1
}
which completes the proof.
\end{proof}

We shall prove Theorem~\ref{thm:tail:general:generation:size} based on the following result.

\begin{propn}[Lemma~3.8 in \cite{jessen:mikosch:2006}] \label{propn:tail:random:sum}
Consider an iid sequence $(X_i : i \ge 1)$ of non-negative random variables independent of the integer-valued non-negative random variable $K$. Define $S_K = \sum_{i=1}^K X_i$. If $K, X_1 > 0$ are regularly varying with indices $\gamma_1 \in [0,1)$ and $\gamma_2 \in [0,1)$ respectively. Then
$$\prob (S_K > x) \sim \prob \Big[ K > \big(\prob(X > x) \big)^{-1} \Big] \sim x^{-\gamma_1 \gamma_2} (L_X(x))^{\gamma_2} L_K(x^{\gamma_1} (L_X(x))^{-1}).$$
\end{propn}

\begin{proof}[Proof of Theorem~\ref{thm:tail:general:generation:size}]
We shall prove the result using induction. Note that $Z_2 \eqd \sum_{i=1}^{Z_1 + 1} B_{2,i}$ where $(B_{2,i}: i \ge 1)$ is a collection of independent copies of $Z_1$. Then we can use Proposition~\ref{propn:tail:random:sum} with $K= Z_1$, $X_i = B_{2,i}$, $\gamma_1 = 0$, $\gamma_2 =0$, $L_K (x) \sim  (\log x)^{-\alpha} L(\log x)$ and $L_X(x) \sim (\log x)^{-\alpha} L(\log x)$ to obtain
\begin{align}
\prob(Z_2  > m) & \sim \prob \Big( Z_1 + 1 > \big( \prob(Z_1 > m) \big)^{-1} \Big) \nonumber \\
& \sim \prob \big(Z_1 > ( \log m)^{\alpha} \textred{(}L(\log m) \textred{)^{-1}} \big) \nonumber \\
& \sim \Big[ \log \Big( (\log m)^\alpha \big( L(\log m) \big)^{-1} \Big) \Big]^{-\alpha} L \Big( \log \big[ (\log m)^{\alpha} \big( L(\log m) \big)^{-1} \big] \Big) \nonumber \\
& = \big( \alpha (\log^{(2)} m) \big)^{-\alpha} \Big[ 1 - \frac{\log L(\log x)}{\alpha \log^{(2)} m} \Big]^{-\alpha} L \Big[ \alpha \textred{(}\log^{(2)} m \textred{)} \Big( 1 - \frac{\log L(\log m)}{\alpha \log^{(2)} m} \Big) \Big] \nonumber \\
& \sim \alpha^{-\alpha} (\log^{(2)} m)^{-\alpha} L( \alpha \log^{(2)} m).
\end{align}
We have used the fact that $\lim_{m \to \infty} \log L( \log m) / \log^{(2)} m  = 0$ which can be proved using Potter's bound given in \cite[Lemma~2 in page~277]{feller:1971}. Hence the result is proved for $l=2$. Note that $Z_{l+1} \eqd \sum_{l=1}^{Z_{l} + 1} B_{l,i}$ where $(B_{l,i} : i \ge 1)$ are independent copies of $Z_1$. We shall assume that \eqref{eq:tail:asymp:generation:size} holds for $Z_l$ with $l \ge 3$. Then we obtain similarly to the previous asymptotics
\begin{align*}
\prob \Big( Z_{l+1} > m \Big) & \sim \prob \Big( Z_{l} + 1 > \big( \prob (Z_1 > m) \big)^{-1} \Big) \nonumber \\
& \sim \alpha^{-\alpha} \Big( \log^{(k+1)} m +  \log^{(k-1)} \alpha - \log^{(k)} L(\log m)  \Big)^{-\alpha}  \nonumber \\
& \hspace{2cm} L \big( \alpha \log^{(k+1)} m +  \alpha\log^{(k-1)} \alpha - \alpha \log^{(k)} L(\log m) \big) \\
& \sim \alpha^{- \alpha} \big( \log^{(k+1)} m \big) L(\log^{(k+1)} m).
\end{align*}
We have again used Potter's bound to show that $\lim_{ m \to \infty} \log^{(k)} L(\log m)/ \log^{(k+1)} m = 0$. Hence we conclude the proof.
\end{proof}

\subsection{A weaker version of Theorem~\ref{thm:tail:general:generation:size}}

\begin{thm} \label{thm:tail:general:generation:size:weaker}
Under the assumptions \eqref{eq:prob:mass:func:Bn}, \eqref{eq:assumption:solomon} and \eqref{assregvar},
\aln{
\limsup_{m \to \infty} \frac{\prob(Z_l \ge m)}{ \Big(\log^{(l)} m \Big)^{-\alpha} L \Big(\log^{(l)} m \Big)} \le \sum_{t=0}^{l-1} \alpha^{-\alpha t} ,
\label{upperin}}
\aln{
\liminf_{m \to \infty} \frac{\prob(Z_l \ge m)}{ \Big(\log^{(l)} m \Big)^{-\alpha} L \Big(\log^{(l)}m \Big)} \ge \alpha^{-\alpha},
\label{lowerin}
}
where $\log^{(l)} m = \underbrace{\log \ldots \log}_{l~~many} m$.
\end{thm}

\begin{proof}[Proof of Theorem~\ref{thm:tail:general:generation:size:weaker}]
The assertion for $l=1$ follows from Theorem~\ref{thm:tail:first:generation:size}.
We will use mathematical induction.
Let us assume then that the assertion of Theorem~\ref{thm:tail:general:generation:size} holds for $l$. We will
prove it for $l+1\ge 2$.
We start from proving upper estimate \eqref{upperin}.
Define
\aln{
g_{l+1}(m) = \left(\frac{\log^{(l)} m}{\log^{(l+1)} m}\right)^\alpha \frac{L\Big(\log^{(l+1)} m \Big)}{L \Big(\log^{(l)} m \Big)}
}
for every $l \ge 2$ and $m \ge 1$.
Our aim is to obtain the tail asymptotics of distribution of
\alns{
Z_{l+1} = \sum_{i=1}^{Z_l + 1} B_{l+1,i},
}
where $(B_{l+1,i} : i \ge 1)$ are independent copies of  geometric random variable $B_{l+1}$ with success probability $A_l$
(its probability mass function is specified in \eqref{eq:prob:mass:func:Bn}) and independent of $Z_l$ conditioned on $\mathcal{A}$.
Thus
\aln{
& \prob(Z_{l+1} > m) \nonumber \\
& = \prob(Z_{l+1} > m, Z_l = 0) + \prob( Z_{l+1} > m, Z_l \ge 1) \nonumber \\
& = \prob(B_{l+1,1} > m, Z_l = 0) + \prob \Big(Z_{l+1} > m, 1 \le Z_l \le [g_{l+1}(m)] \Big) + \prob \Big( Z_{l+1} > m, Z_l \ge [g_{l+1}(m)] \Big) \nonumber \\
& = \prob( B_{l+1} > m) \prob(Z_l = 0) + \prob \Big(Z_{l+1} > m, 1 \le Z_l \le [g_{l+1}(m)] \Big) + \prob \Big( Z_{l+1} > m, Z_l \ge [g_{l+1}(m)] \Big) \label{eq:tail:ztwo:app1:disp1}
}
conditioning on $\mathcal{A}$ and using the fact that $Z_l$ and $(B_{l+1,i} : i \ge 1)$ are independent conditioned on $\mathcal{A}$. Now note that the marginal distribution of $B_{l+1}$ is same as that of $Z_1$ and $\prob(Z_l = 0) = \exptn(A) $. Thus the first term in \eqref{eq:tail:ztwo:app1:disp1} becomes
\aln{
\prob(Z_1 > m) \exptn(A)
}
and its asymptototics by Theorem~\ref{thm:tail:first:generation:size} is of order
$(\log m)^{-\alpha} L(\log m)$, hence is negligible with respect of the main postulated asymptotics.
Moreover, the remaining terms in \eqref{eq:tail:ztwo:app1:disp1} gives:
\aln{
& \prob \Big( Z_{l+1} > m, 1 \le Z_l \le [g_{l+1}(m)] \Big) + \prob( Z_l > [g_{l+1}(m)]) \nonumber \\
& = \sum_{k=1}^{[g_{l+1}(m)]} \prob \Big( \sum_{i=1}^{k+1} B_{l+1, i} \ge m \Big) \prob ( Z_l =k ) + \prob (Z_l > [g_{l+1}(m)]) \label{eq:tail:ztwo:app1:disp2}
}
using again the conditioning argument for the terms inside the sum.
The upper asymptotics of the last term in \eqref{eq:tail:ztwo:app1:disp2} follows from the induction assumption:
\aln{
\limsup_{m \to \infty} \frac{\prob(Z_l \ge [g_{l+1}(m)])}{ \Big(\log^{(l)} [g_{l+1}(m)] \Big)^{-\alpha} L \Big(\log^{(l)} [g_{l+1}(m)] \Big)} \le \sum_{t=0}^{l-1} \alpha^{-\alpha t} .
\label{dodane}
}
We derive the upper bound for the sum in \eqref{eq:tail:ztwo:app1:disp2} as follows:
\aln{
& \sum_{k=1}^{[g_{l+1}(m)]} \prob \Big( \sum_{i=1}^{k+1} B_{l+1,i} > m \Big) \prob(Z_l = k) \nonumber \\
& \le \sum_{k=1}^{[g_{l+1}(m)]} \prob \bigg( \bigcup \Big( B_{l+1,i} > \frac{m}{k+1} \Big)  \bigg) \prob (Z_l = k) \nonumber \\
& \le \sum_{k=1}^{[g_{l+1}(m)] + 1} \sum_{i=1}^{k+1} \prob \Big( B_{l+1,i} > \frac{m}{k+1} \Big) \prob(Z_l = k) \nonumber \\
& = \sum_{k=1}^{[g_{l+1}(m)]} (k+1) \prob \Big( B_{l+1} > \frac{m}{k+1} \Big) \prob(Z_l  = k) \label{eq:tail:ztwo:app1:disp3}
}
using the fact that $B_{l+1,i} \eqd B_{l+1}$ for every $i \ge 1$.
Note that $k < [g_{l+1}(m)]$ implies that $(k+1)^\inv < ([g_{l+1}(m)] + 1)^\inv$. Thus
$\prob \Big(B_{l+1} > m (k+1)^\inv \Big) \le \prob \Big( B_{l+1} > m ([g_{l+1}(m)] + 1)^\inv \Big)$ and
we obtain the following upper bound for the sum in \eqref{eq:tail:ztwo:app1:disp3}:
\aln{
& \prob \Big( B_{l+1} > \frac{m}{[g_{l+1}(m)] + 1} \Big) \sum_{k=1}^{[g_{l+1}(m)]} (k+1) \prob(Z_l = k) \nonumber \\
&\le [g_{l+1}(m) + 1] \prob \Big( Z_1 > \frac{m}{[g_{l+1}(m) +1]} \Big),\label{eq:tail:ztwo:app1:disp4}
}
where we use the fact that the marginal distribution of $B_{l+1}$ and $Z_1$ are same.
Thus by \eqref{eq:tail:ztwo:app1:disp1} and \eqref{dodane}, the asymptotic tail of the right hand side of $\prob(Z_{l+1} > m)$
can be dominated by:
\aln{
& \bigg[ g_{l+1}(m)
+ 1 \bigg] \Big( \log m - \log [g_{l+1}(m) + 1] \Big)^{-\alpha}L \Big( \log m - \log [g_{l+1}(m) + 1] \Big)\nonumber \\
 & \hspace{2cm}   + (1 + \alpha^{-\alpha})\Big( \log [g_{l+1}(m)] \Big)^{-\alpha} L \big(\log [g_{l+1}(m)] \big) \nonumber \\
& = \big( \log^{(l)} m \big)^{\alpha} \Big( L (\log^{(l)} m) \Big)^\inv  \bigg[ \Big( \log^{(l+1)} m \Big)^{-\alpha} L \big( \log^{(l+1)} m \big)
+ \big( \log^{(l)} m \big)^{-\alpha} L(\log^{(l)} m) \bigg] \label{eq:ztwo:tail:app1:disp6} \\
& \big(\log m )^{-\alpha} \bigg( 1 - \frac{\log [g_{l+1}(m) + 1]}{\log m} \bigg)^{-\alpha}
\frac{L \Big( \log m \big( 1- (\log [g_{l+1}(m) +1]) (\log m)^\inv \big)\Big)}{L(\log m)}  L(\log m) \label{eq:ztwo:tail:app1:disp7} \\
&  + \sum_{t=0}^{l-1} \alpha^{-\alpha t} \bigg( \frac{\log [g_{l+1}(m)]}{\log^{l+1} m} \bigg)^{- \alpha}  \frac{ L \Big( \frac{\log[g_{l+1}(m)]}{\log^{(l+1)} m}  \log^{(l+1)} m\Big)}{ L(\log^{(l+1)} m)} \big(\log^{(l+1)} m )^{-\alpha} L (\log^{(l+1)} m). \label{eq:ztwo:tail:app1:disp8}
}
Note that
\aln{
\lim_{m \to \infty} \frac{\log [g_{l+1}(m)]}{\log^{(l+1)} m} = \alpha. \label{eq:asymptotic:log:gm}
}
Moreover, the first term and the second term in the product \eqref{eq:ztwo:tail:app1:disp6} cancels with the the first term and the last term in \eqref{eq:ztwo:tail:app1:disp7}. Using \eqref{eq:asymptotic:log:gm}, one can observe that the second and third term  in \eqref{eq:ztwo:tail:app1:disp7} converges to $1$ as $m \to \infty$.
Note also that \eqref{eq:asymptotic:log:gm} leads to
the conclusion that the third term in \eqref{eq:ztwo:tail:app1:disp6} for the fist summand  and last increment \eqref{eq:ztwo:tail:app1:disp8}
behaves asymptotically like $\big(\log^{(l+1)} m )^{-\alpha} L (\log^{(l+1)} m)$ and
$$\sum_{t=0}^{l-1} \alpha^{-\alpha t} \alpha^{-\alpha}\big(\log^{(l+1)} m )^{-\alpha} L (\log^{(l+1)} m)=
\sum_{t=1}^{l} \alpha^{-\alpha t} \big(\log^{(l+1)} m )^{-\alpha} L (\log^{(l+1)} m),$$ respectively.
This completes the proof of the upper bound.

We will prove now the lower estimate \eqref{lowerin}.
We start from basic for our purposes lemma.
\begin{propn}\label{eq:nagaev:lower:bound:iid}
Let $(X_i : i \ge 1)$ be a collection of independent random variables with geomtric distribution with parameter $q$, that is $\prob(X>x)=q^x$ for generic random variable $X$. Let $\delta > 0$. Then
for
\begin{equation}\label{warunek}
x\ge (\log1/q)^{-1}\Big(\log n-\log \frac{\delta}{1-\delta}\Big)
\end{equation} we have
\aln{
\prob \Big( \sum_{i=1}^n X_i > x \Big) \ge \frac{1}{1-\delta}n \prob( X > x).
}
\end{propn}

\begin{proof}
Note that for $y\in (0,1)$ we have
\[(1-y)^n\le 1-ny +\frac{n(n-1)}{2}y^2.\]
Hence
\begin{eqnarray*}
\lefteqn{\prob \Big( \sum_{i=1}^n X_i > x \Big) =1- \prob \Big( \sum_{i=1}^n X_i \le  x \Big)\ge 1- \prod_{i=1}^n\prob \Big(X_i \le  x \Big)}\\&&
=1-\Big(1-\prob \big(X_i >  x\big) \Big)^n\ge n\prob \big(X_i >  x\big) -\frac{n(n-1)}{2}\prob \big(X_i >  x\big)^2.
\end{eqnarray*}
Now \eqref{warunek} produces $(n-1)\prob \big(X_i >  x\big)\le n\prob \big(X_i >  x\big)\le \frac{\delta}{1-\delta}$ and thus the proof is completed.
\end{proof}

Let
\[h_{0}(m) = \frac{\Big( \log m \Big)^\alpha}{L \Big( \log m \Big)}  \quad\mbox{and}\qquad
h_{1}(m) = \frac{\Big( \log m \Big)^\alpha}{L \Big( \log m \Big)} m.\]
Observe now that
\aln{
\prob(Z_{l+1} > m) & = \prob \Big( Z_{l+1} > m, Z_l \le [h_1(m)] \Big) + \prob \Big( Z_{l+1} > m, Z_l > [h_1(m)] \Big) \nonumber \\
& \ge \prob \Big( Z_{l+1} > m, Z_l \le [h_1(m)] \Big) \nonumber \\
& = \sum_{k=0}^{[h_1(m)]} \prob \Big( \sum_{j=1}^{k+1} B_{l+1,j} \ge m, Z_l = k \Big) \nonumber \\
& =  \prob (B_{l+1} \ge m) \prob(Z_l = 0) + \sum_{k=1}^{[h_1(m)]} \prob \Big( \sum_{j=1}^{k+1} B_{l+1,j} \ge m \Big) \prob(Z_l =k) \nonumber\\
& = \prob(B_{l+1} \ge  m) \prob(Z_l = 0) + \sum_{k=1}^{[h_1(m)]} \exptn \bigg[ \prob \bigg( \sum_{j=1}^{k+1} B_{l+1,j} \ge m \bigg| \mathcal{A} \bigg) \bigg] \prob(Z_l = k).  \label{eq:lower:bound:ztwo:disp1}
}
Recall that $(B_{l+1,j} : j \ge 1)$ are independent geometric random variables with success probability $A_{l+1}$ conditioned on $\mathcal{A}$.
Furthermore, there exists an $m_0 \in \mbbn$ such that  $m  > (\log1/(1-A_{l+1}))^{-1}(\log k-\log \frac{\delta}{1-\delta}) $ for some $\delta >0$ all $m > m_0$
and $k\le h_1(m)$. We choose $m$ sufficiently large to have $h_0(m)>m_0$.
Thus from Proposition \ref{eq:nagaev:lower:bound:iid} it follow that the expression of the right hand side of \eqref{eq:lower:bound:ztwo:disp1}
can be estimated from below by
\aln{
& \prob (B_{l+1} \ge m) \prob(Z_l = 0) + \sum_{k=1}^{[h_0(m)]} \prob \Big( \sum_{j=1}^{k+1} B_{l+1,j} \ge m \Big) \prob(Z_l =k)\\& \qquad+ \frac{1}{1-\delta} \prob(B_{l+1} > m)\sum_{k=[h_0(m)] +1}^{[h_1(m)]} k \prob( Z_l = k) \nonumber \\
& \ge \frac{1}{1-\delta} \prob(B_{l+1} > m)\sum_{k=[h_0(m)] +1}^{[h_1(m)]} k \prob( Z_l = k). \label{eq:lower:bound:ztwo:disp4}
}
By Theorem~\ref{thm:tail:first:generation:size}, for any $\epsilon >0$ there exists a large enough $m_1 \in \mbbn$ such that for all $m > m_1$,  we have
\aln{
\prob(B_{l+1} > m)=\prob(Z_1 > m) \ge (1- \epsilon)  \Big( \log m \Big)^{-\alpha} L(\log m). \label{eq:lower:bound:ztwo:disp5}
}
Note that by integration by parts formula
for the tail $\overline{F}(x)=1-F(x)$ of a distribution function $F$ supported on positive half-line, we have
\begin{equation}\label{newibp}
-\int_{x_0}^x y \overline{F}^\prime(y)dy=-x\overline{F}(x)+\int_{x_0}^x\overline{F}(y)dy+x_0\overline{F}(x_0)\ge x_0\overline{F}(x_0)\left(1-\frac{\overline{F}(x)}{\overline{F}(x_0)}\right).\end{equation}
Since $h_0(m) \uparrow \infty$ and $h_1(m)/h_0(m) \uparrow \infty$ as $m \to \infty$,
using the integral test for series and \eqref{newibp},
we can conclude the existence of $m_2 \in \mbbn$ such that for all $m>m_2$,
\aln{
\sum_{k=[h_0(m)] +1}^{[h_1(m)]} k \prob( Z_l = k) & \ge (1- \epsilon)C_l h_0(m)\big( \log^{(l)} h_0(m)\big)^{-\alpha}  L(\log^{(l)} h_0(m)),
\label{eq:lower:second:disp1}
}
where $C_1=1$ by Theorem~\ref{thm:tail:first:generation:size} and $C_l=\alpha^{-\alpha}$ for $l\ge 2$ by induction assumption.
Hence from \eqref{eq:lower:bound:ztwo:disp1}, \eqref{eq:lower:bound:ztwo:disp5} and \eqref{eq:lower:second:disp1} for $m>\max\{m_1,m_2\}$ we have
\aln{
\prob(Z_{l+1} > m) & \ge (1-\epsilon)^2 \frac{1}{1-\delta} C_l \Big( \log m \Big)^{-\alpha} L(\log m) h_0(m)\nonumber\\
&\qquad\big( \log^{(l)} h_0(m)\big)^{-\alpha}
L(\log^{(l)} h_0(m)).\label{koniecrozwazan}
}
Now, if $l=1$ then $\log^{(l)} h_0(m)=\log h_0(m)$ is of order $\alpha \log^{(2)} m$ and
then the right hand side of \eqref{koniecrozwazan} is of order
$\frac{1}{1-\delta}\alpha^{-\alpha}(\log^{(2)}m)^{-\alpha}L(\log^{(2)}m)$. Taking $\delta \downarrow 0$, this also gives that
$C_{1+1}=C_2=\alpha^{-\alpha}$.
If $l\ge 0$ then  $\log^{(l)} h_0(m)$ is of order $\log^{(l+1)} m$ and
then the right hand side of \eqref{koniecrozwazan} is of order
$\alpha^{-\alpha}(\log^{(l+1)}m)^{-\alpha}L(\log^{(l+1)}m)$.
Letting $\epsilon \to 0$ completes the proof.
\end{proof}

\noindent{\bf{Acknowledgements.}} This research was partially supported by National Science Centre, Poland, under the grant 2018/29/B/ST1/00756. A B was supported  by Dutch Science foundation  NWO VICI grant \#
639.033.413. Authors are also thankful to Krishanu Maulik and Marcin Preisner for giving valuable suggestions. The authors are also thankful to the anonymous referees for careful reading and their valuable comments.\\

\bibliographystyle{abbrvnat}			

\end{document}